\documentclass[11pt]{article}

\usepackage{umvan4_s3_lib}

\begin{document}

\title{\sc Super Vertex Algebras, Meromorphic Jacobi Forms and Umbral Moonshine}

\author[1]{John F. R. Duncan\thanks{Email: \texttt{john.duncan@emory.edu}}}
\author[2]{Andrew O'Desky\thanks{Email: \texttt{aodesky@umich.edu}}}
\affil[1]{Department of Mathematics and Computer Science, Emory University, Atlanta, GA 30322, U.S.A.}
\affil[2]{Department of Mathematics, University of Michigan, Ann Arbor, MI 48109, U.S.A.}
\date{}

\renewcommand{\thefootnote}{\fnsymbol{footnote}} 
\footnotetext{\emph{MSC2010:} 11F37, 11F50, 17B69, 17B81, 20C35.}     
%17B69:Vertex operators; vertex operator algebras and related structures
%17B81:Applications to physics
%20C34:Representations of sporadic groups
%20C35:Applications of group representations to physics
%11F11:Holomorphic modular forms of integral weight
%11F27:Theta series; Weil representation; theta correspondences
%11F37:Forms of half-integer weight; nonholomorphic modular forms
%11F50:Jacobi forms

\maketitle

\begin{abstract}
The vector-valued mock modular forms of umbral moonshine may be repackaged into meromorphic Jacobi forms of weight one. In this work we constructively solve two cases of the meromorphic module problem for umbral moonshine. Specifically, for the type A Niemeier root systems with Coxeter numbers seven and thirteen, 
we construct corresponding bigraded super vertex operator algebras, equip them with actions of the corresponding umbral groups, and verify that the resulting trace functions on canonically twisted modules recover the meromorphic Jacobi forms that are specified by umbral moonshine. We also obtain partial solutions to the meromorphic module problem for the type A Niemeier root systems with Coxeter numbers four and five, by constructing super vertex operator algebras that recover the meromorphic Jacobi forms attached to maximal subgroups of the corresponding umbral groups. 
\end{abstract}

\clearpage

\setcounter{tocdepth}{1}

\tableofcontents

%----------------------------------------------------------------------------------------
\section{Introduction}\label{sec:intro}
%----------------------------------------------------------------------------------------

Eguchi--Ooguri--Tachikawa initiated a new phase in moonshine with their 
observation \cite{Eguchi2010} that representations of the largest sporadic Mathieu group $M_{24}$ are visible in the 
multiplicities of irreducible superconformal algebra modules in the K3 elliptic genus. 
The generating function of these multiplicities is a mock modular form $H^{(2)}$ of weight $\frac12$ (cf. \cite{Dabholkar:2012nd}). 
Once {\em twined} counterparts $H^{(2)}_g$ for $g\in M_{24}$ had been identified \cite{MR2793423,Gaberdiel2010a,Gaberdiel2010,Eguchi2010a} and characterized \cite{Cheng2011}, Gannon was able to confirm \cite{MR3539377} that there is a corresponding graded $M_{24}$-module, for which the $q$-series of $H^{(2)}=H^{(2)}_e$ is the graded dimension. 
But so far there has been no explicit construction of this {\em Mathieu moonshine} module, such as might be compared to the vertex operator algebra of monstrous moonshine \cite{MR554399,Tho_FinGpsModFns,Tho_NmrlgyMonsEllModFn} that was discovered by Frenkel--Lepowsky--Meurman \cite{FLMPNAS,FLMBerk,FLM}, and used to prove the monstrous moonshine conjectures by Borcherds \cite{MR1172696}. The purpose of this paper is to solve a closely related construction problem, for some closely related instances of moonshine.

To motivate our approach we recall the 
curious 
circumstance that the
{\em McKay--Thompson series} $H^{(2)}_g$ 
of Mathieu moonshine
may be repackaged into modular forms of different kinds. Indeed, 
if $\chi^{(2)}_g$ is the number of fixed points of $g\in M_{24}$ in the defining permutation representation on $24$ points, then 
\begin{gather}\label{eqn:intro-Z2g}
	Z^{(2)}_g(\tau,z):=
	\chi^{(2)}_g\frac{\mu_{2,0}(\tau,z)}{\mu_{1,0}(\tau,z)}
	+H^{(2)}_g(\tau)\frac{\theta_1(\tau,z)^2}{\eta(\tau)^3}
\end{gather}
is a weak Jacobi form of weight $0$, index $1$, and some level depending on $g$ (where $\mu_{m,0}$ is defined in (\ref{eqn:jf-mum0}), and $\theta_1$ and $\eta$ are recalled in \S\ref{sec:ujf}). 
The $g=e$ case of 
(\ref{eqn:intro-Z2g}) expresses the K3 elliptic genus in terms of $H^{(2)}$, and is the starting point of \cite{Eguchi2010}. 

This suggests that the Mathieu moonshine module might be realized in terms of a suitably chosen K3 sigma model, but it was found in \cite{MR2955931} that the symmetries of these objects are precisely the subgroups of the automorphism group of the Leech lattice---i.e., the {\em Conway group}, $\Co_0\simeq 2.\Co_1$---that fix a $4$-space. In particular, $M_{24}$ does not appear. Interestingly, 
it has been found \cite{MR3465528} that suitable trace functions attached to the moonshine module for Conway's group (see \cite{Dun_VACo,MR3376736}) attach weak Jacobi forms of weight $0$ and index $1$ (with level) to $4$-space-fixing automorphisms of the Leech lattice, and this construction recovers the K3 elliptic genus when applied to the trivial symmetry.  More generally, many---but not all---of the $Z^{(2)}_g$ appear in this way. So the Conway moonshine module serves as a kind of ``fake'' Mathieu moonshine module, with a closer connection to K3 sigma models (cf. \cite{Cheng:2016org,2017arXiv170205095P,2017arXiv170403678C,2017arXiv170403813T}) than to Mathieu moonshine itself.

As an alternative to (\ref{eqn:intro-Z2g}) we may consider the functions $\psi^{(2)}_g:=-\mu_{1,0}Z^{(2)}_g$, which are meromorphic Jacobi forms of weight $1$ and index $2$ (cf. \S\ref{sec:jf}).
Although this is a simple manipulation 
it seems to be essential for {\em umbral moonshine} 
\cite{UM,MUM,mumcor}, since in this more general setting weak Jacobi form formulations of the McKay--Thompson series are only known in some cases, whereas meromorphic Jacobi forms $\psi^{(\ell)}_g$ may be constructed in a uniform way (cf. \S4 of \cite{MUM}, or \S\ref{sec:jf} of this work). In umbral moonshine 
vector-valued mock modular forms $H^{(\ell)}_g=(H^{(\ell)}_{g,r})$ are associated to (outer) automorphisms of {\em Niemeier lattices} (i.e., self-dual even positive definite lattices of rank $24$ with roots), and Mathieu moonshine is recovered by specializing to the Niemeier lattice whose root system is $A_1^{\oplus 24}$. It has been proven \cite{umrec} that the 
$H^{(\ell)}_g$ 
define modules for the groups to which they are attached, but except for the case of the Niemeier lattice $E_8^{\oplus 3}$ (see \cite{MR3649360}), no explicit constructions of the $H^{(\ell)}_g$ as traces on algebraic structures are known. 

An extension of the method of \cite{MR3649360} apparently requires a finer knowledge of the relationship between mock modular forms and indefinite lattices than is currently available. So here we promote the alternative approach of focusing on the meromorphic Jacobi forms $\psi^{(\ell)}_g$ rather than the vector-valued mock modular forms $H^{(\ell)}_g$. We call this the {\em meromorphic module problem} for umbral moonshine. 

In this work we solve the meromorphic module problem for the cases of umbral moonshine corresponding to the Niemeier lattices $A_{6}^{\oplus 4}$ and $A_{12}^{\oplus 2}$ (corresponding to $\ell=7$ and $\ell=13$, respectively), and provide partial solutions for $A_{3}^{\oplus 8}$ and $A_{4}^{\oplus 6}$ (corresponding to $\ell=4$ and $\ell=5$, respectively). We achieve this by considering suitable tensor products of simple free field super vertex operator algebras, equipping them with suitable bigradings, and identifying suitable trace functions on their canonically twisted modules. This approach is motivated by the fact that many of the corresponding meromorphic Jacobi forms admit product formulas (cf. \S\ref{sec:ujf}). For $A_{6}^{\oplus 4}$ and $A_{12}^{\oplus 2}$ the corresponding umbral groups act naturally, and all the corresponding $\psi^{(\ell)}_g$ are realized explicitly (see Theorems \ref{thm:moon:7-psig} and \ref{thm:moon:13-psig}). For $A_{3}^{\oplus 8}$ and $A_{4}^{\oplus 6}$ we find actions of certain maximal subgroups of the corresponding umbral groups, and realize most, but not all, of the corresponding $\psi^{(\ell)}_g$ (see Propositions \ref{prop:moon:4-psig} and \ref{prop:moon:5-psig}). 

It will be interesting to see if a modification of the methods presented here can solve the meromorphic module problem completely for $A_{3}^{\oplus 8}$ and $A_{4}^{\oplus 6}$. We expect that that would yield some useful insight into the broader question of constructing meromorphic umbral moonshine modules in general.

The structure of the article is as follows. In \S\ref{sec:va} we briefly recall the Clifford module and Weyl module constructions of super vertex operator algebras and their canonically twisted modules. In \S\ref{sec:jf} we recall the relationship between the mock modular forms $H^{(\ell)}_g$ and the meromorphic Jacobi forms $\psi^{(\ell)}_g$, for the Niemeier lattices with root system of the form $A_{\ell-1}^{\oplus d}$ (i.e., the cases that $\ell-1$ is a divisor of $24$). Our new results are Theorems \ref{thm:moon:7-psig} and \ref{thm:moon:13-psig}, and Propositions \ref{prop:moon:4-psig} and \ref{prop:moon:5-psig}. They appear in \S\ref{sec:moon}. In \S\ref{sec:tab} we present the character tables of the umbral groups $G^{(\ell)}$ for $\ell\in \{4,5,7,13\}$, and for the relevant maximal subgroups in case $\ell\in \{4,5\}$. In \S\ref{sec:ujf} we recall the explicit expressions for the $\psi^{(\ell)}_g$ that were used for the purpose of proving the (abstract) module conjectures for umbral moonshine in \cite{umrec}. These expressions play a role in the proofs of 
our results in \S\ref{sec:moon}. In \S\ref{sec:tab:eul} we recall the definitions of the characters $\chi^{(\ell)}_g$ and $\bar\chi^{(\ell)}_g$ which appear in the formula that relates $H^{(\ell)}_g$ to $\psi^{(\ell)}_g$ (cf. \S\ref{sec:jf}).

%----------------------------------------------------------------------------------------
\section{Super Vertex Algebras}\label{sec:va}
%----------------------------------------------------------------------------------------

We briefly review the Clifford module and Weyl module constructions of super vertex operator algebras, and their canonically twisted modules in this section. 
The umbral moonshine modules we present in \S\ref{sec:moon} will be realized as tensor products of these simple free field super vertex operator algebras. 
We refer the reader to \cite{MR1651389,MR2023933,MR2082709} for background on vertex algebra theory.

\subsection{Clifford Modules}\label{sec:va:cliff}

Let $\gt{a}$ be a complex vector space and let $\lab\cdot\,,\cdot\rab$ be a non-degenerate symmetric bilinear form on $\gt{a}$. The {\em Clifford algebra} associated to this data is $\Cliff{\gt{a}}:=T(\gt{a})/I$ where $T(\gt{a}):=\CC {\bf 1}\oplus \gt{a}\oplus\gt{a}^{\otimes 2}\oplus\cdots$ is the tensor algebra of $\gt{a}$, and $I$ is the ideal of $T(\gt{a})$ generated by the expressions $a\otimes a'+a'\otimes a -\lab a,a'\rab{\bf 1}$ for $a,a'\in\gt{a}$. The composition of natural maps $\gt{a}\to T(\gt{a})\to \Cliff{\gt{a}}$ is an embedding, so we may regard $\gt{a}$ as a subspace of $\Cliff{\gt{a}}$. Let ${\bf 1}$ also denote the unit in $\Cliff{\gt{a}}$. A {\em polarization} of $\gt{a}$ is a vector space splitting $\gt{a}=\gt{a}^+\oplus \gt{a}^-$ for which the summands $\gt{a}^\pm$ are isotropic for the given bilinear form. Given such a splitting (this requires $\dim \gt{a}$ to be even if it is finite) the induced module $\Cliff{\gt{a}}\otimes_{\lab \gt{a}^+\rab}\CC\vv$ is irreducible for $\Cliff{\gt{a}}$, when $\lab\gt{a}^+\rab$ is the sub algebra of $\Cliff{\gt{a}}$ generated by ${\bf 1}$ and $\gt{a}^+$, and $\CC\vv$ is the unique unital $\lab \gt{a}^+\rab$-module such that $a\vv=0$ for every $a\in \gt{a}^+$.

Henceforth assume that $\dim\gt{a}$ is finite and even. For $r\in \frac12\ZZ$ let $\gt{a}(r)$ be a vector space isomorphic to $\gt{a}$. Choose an isomorphism $\gt{a}\to \gt{a}(r)$ for each $r$, and denote it $a\mapsto a(r)$. Define $\hat{\gt{a}}:=\bigoplus_{n\in\ZZ}\gt{a}(n+\frac12)$ and $\hat{\gt{a}}_\tw:=\bigoplus_{n\in\ZZ}\gt{a}(n)$, and extend $\lab\cdot\,,\cdot\rab$ to $\hat{\gt{a}}$ and $\hat{\gt{a}}_\tw$ by requiring that $\lab a(r),a'(r')\rab=\lab a,a'\rab\delta_{r+r',0}$ for $a,a'\in\gt{a}$ and $r,r'\in\frac12\ZZ$. Choose a polarization $\gt{a}=\gt{a}^+\oplus \gt{a}^-$ of $\gt{a}$, and define polarizations of $\hat{\gt{a}}$ and $\hat{\gt{a}}_\tw$ by setting 
\begin{gather}
\begin{split}
	\hat{\gt{a}}^+:=\bigoplus_{n\geq 0
	}\gt{a}(n+\tfrac12),&\qquad
	\hat{\gt{a}}^-:=\bigoplus_{n<0
	}\gt{a}(n+\tfrac12),\\
	\hat{\gt{a}}_\tw^+:=\gt{a}^+(0)\oplus\bigoplus_{n>0
	}\gt{a}(n),&\qquad
	\hat{\gt{a}}_\tw^-:=\gt{a}^-(0)\oplus\bigoplus_{n<0
	}\gt{a}(n).	
\end{split}
\end{gather}
The {\em Clifford module super vertex algebra} associated to $\gt{a}$ and $\lab\cdot\,,\cdot\rab$ is the unique super vertex algebra structure on $A(\gt{a}):=\Cliff{\hat{\gt{a}}}\otimes_{\lab\hat{\gt{a}}^+\rab}\CC\vv$ such that $\vv$ is the vacuum, and $Y(a(-\frac12)\vv,z)=\sum_{n\in\ZZ} a(n+\frac12)z^{-n-1}$ for $a\in \gt{a}$. Note that $A(\gt{a})$ is simple. 
Define $A(\gt{a})_\tw:=\Cliff{\hat{\gt{a}}_\tw}\otimes_{\lab\hat{\gt{a}}_\tw^+\rab}\CC\vv_\tw$ (where $\CC\vv_\tw$ is the unique unital $\lab\hat{\gt{a}}_\tw^+\rab$-module such that $u\vv_\tw=0$ for $u\in\hat{\gt{a}}_\tw^+$). Then there is a unique structure of canonically twisted $A(\gt{a})$-module on $A(\gt{a})_\tw$ such that $Y_\tw(a(-\frac12),z)=\sum_{n\in\ZZ}a(n)z^{-n-\frac12}$ for $a\in \gt{a}$. If $\{a_i^\pm\}$ is a basis for $\gt{a}^\pm$ such that $\lab a_i^\mp,a_j^\pm\rab=\delta_{i,j}$ then 
\begin{gather}
\omega:=\frac12\sum_i ( a_i^+(-\tfrac32)a_i^-(-\tfrac12) - a_i^+(-\tfrac12)a_i^-(-\tfrac32) )\vv
\end{gather}
is a Virasoro element for $A(\gt{a})$ with central charge $c=\frac12\dim\gt{a}$ that makes $A(\gt{a})$ a super vertex operator algebra. 

Set $\jmath:=\sum_i a_i^+(-\tfrac12)a_i^-(-\tfrac12)\vv$. Write $J(n)$ for the coefficient of $z^{-n-1}$ in $Y(\jmath,z)$ or $Y_\tw(\jmath,z)$, and write $L(n)$ for the coefficient of $z^{-n-2}$ in $Y(\omega,z)$ or $Y_\tw(\omega,z)$. Then 
$J(0)$ and $L(0)$ commute, and act semisimply on $A(\gt{a})$ and $A(\gt{a})_\tw$, with finite-dimensional (simultaneous) eigenspaces. For the corresponding bigraded dimensions we have 
\begin{gather}
	\tr(y^{J(0)}q^{L(0)-\frac{c}{24}}|A(\gt{a}))=q^{-\frac{d}{48}}\prod_{n>0}(1+y^{-1}q^{n-\frac12})^{\frac d2 }(1+yq^{n-\frac12})^{\frac d2},\\
	\tr(y^{J(0)}q^{L(0)-\frac{c}{24}}|A(\gt{a})_\tw)=y^{\frac d 4}q^{\frac{d}{24}}\prod_{n>0}(1+y^{-1}q^{n-1})^{\frac d2 }(1+yq^{n})^{\frac d2},
\end{gather}
when $d=\dim\gt{a}$. Note that $\omega$ does not depend upon the choice of polarization $\gt{a}=\gt{a}^+\oplus \gt{a}^-$, but $\jmath$ does.

The group $\GL(\gt{a}^+)$ acts naturally on $A(\gt{a})$ and $A(\gt{a})_\tw$, preserving $\omega$ and $\jmath$. For $g\in \GL(\gt{a}^+)$ there is a unique $g'\in \GL(\gt{a}^-)$ such that $\lab ga,g'a'\rab=\lab a,a'\rab$ for all $a\in\gt{a}^+$ and $a'\in \gt{a}^-$. Abusing notation slightly, we write $g$ also for the linear automorphism $(g,g')$ on $\gt{a}=\gt{a}^+\oplus \gt{a}^-$. Then the action of $\GL(\gt{a}^+)$ on $A(\gt{a})$ is given by $g\cdot a_1(r_1)\dots a_n(r_n)\vv:=(ga_1)(r_1)\dots (ga_n)(r_n)\vv$ for $a_i\in\gt{a}$ and $r_i\in \ZZ+\frac12$, and similarly for $A(\gt{a})_\tw$. We then have $Y(gu,z)gv=gY(u,z)v$ and $Y_\tw(gu,z)gw=gY_\tw(u,z)w$ for $u,v\in A(\gt{a})$ and $w\in A(\gt{a})_\tw$, and also $g\omega=\omega$ and $g\jmath=\jmath$, so the bigradings of $A(\gt{a})$ and $A(\gt{a})_\tw$ are preserved.

\subsection{Weyl Modules}\label{sec:va:weyl}

The Weyl module construction runs in parallel with that of the previous section, but with an anti-symmetric bilinear form in place of a symmetric one. So let $\gt{b}$ be a complex vector space and let $\llab\cdot\,,\cdot\rab$ be a non-degenerate anti-symmetric bilinear form on $\gt{b}$. The {\em Weyl algebra} associated to this data is $\Weyl{\gt{b}}:=T(\gt{b})/I$ where 
$I$ is the ideal of $T(\gt{b})$ generated by $b\otimes b'-b'\otimes b -\llab b,b'\rab{\bf 1}$ for $b,b'\in\gt{b}$. Just as for Clifford algebras we may naturally identify $\gt{b}$ as a subspace of $\Weyl{\gt{b}}$, and we write ${\bf 1}$ also for the unit in $\Weyl{\gt{b}}$. 
Given a polarization $\gt{b}=\gt{b}^+\oplus \gt{b}^-$ (so that $\gt{b}^\pm$ is isotropic for $\llab\cdot\,,\cdot\rab$), the induced module $\Weyl{\gt{b}}\otimes_{\lab \gt{b}^+\rab}\CC\vv$ is irreducible for $\Weyl{\gt{b}}$. 

Assume now that $\dim\gt{b}$ is finite. This forces $\dim\gt{b}$ to be even. 
Define $\hat{\gt{b}}:=\bigoplus_{n\in\ZZ}\gt{b}(n+\frac12)$ and $\hat{\gt{b}}_\tw:=\bigoplus_{n\in\ZZ}\gt{b}(n)$, just as in the previous section, and extend $\llab\cdot\,,\cdot\rab$ to $\hat{\gt{b}}$ and $\hat{\gt{b}}_\tw$ by requiring that $\llab b(r),b'(r')\rab=\llab b,b'\rab\delta_{r+r',0}$ for $b,b'\in\gt{b}$ and $r,r'\in\frac12\ZZ$. Choose a polarization $\gt{b}=\gt{b}^+\oplus \gt{b}^-$ of $\gt{b}$, and define polarizations of $\hat{\gt{b}}$ and $\hat{\gt{b}}_\tw$ by setting 
\begin{gather}
\begin{split}
	\hat{\gt{b}}^+:=\bigoplus_{n\geq 0 }\gt{b}(n+\tfrac12),&\qquad
	\hat{\gt{b}}^-:=\bigoplus_{n<0}\gt{b}(n+\tfrac12),\\
	\hat{\gt{b}}_\tw^+:=\gt{b}^+(0)\oplus\bigoplus_{n> 0}
	\gt{b}(n),&\qquad
	\hat{\gt{b}}_\tw^-:=\gt{b}^-(0)\oplus\bigoplus_{n< 0}
	\gt{b}(n).
\end{split}
\end{gather}
The {\em Weyl module super vertex algebra} associated to $\gt{b}$ and $\llab\cdot\,,\cdot\rab$ is the unique super vertex algebra structure on $\uA(\gt{b}):=\Weyl{\hat{\gt{b}}}\otimes_{\lab\hat{\gt{b}}^+\rab}\CC\vv$ such that $\vv$ is the vacuum, and $Y(b(-\frac12)\vv,z)=\sum_{n\in\ZZ} b(n+\frac12)z^{-n-1}$ for $b\in \gt{b}$. Define ${\uA(\gt{b})}_\tw:=\Weyl{\hat{\gt{b}}_\tw}\otimes_{\lab\hat{\gt{b}}_\tw^+\rab}\CC\vv_\tw$. 
Then there is a unique structure of canonically twisted ${\uA(\gt{b})}$-module on ${\uA(\gt{b})}_\tw$ such that $Y_\tw(b(-\frac12),z)=\sum_{n\in\ZZ}b(n)z^{-n-\frac12}$ for $b\in \gt{b}$. If $\{b_i^\pm\}$ is a basis for $\gt{b}^\pm$ such that $\llab b_i^\mp,b_j^\pm\rab=\pm\delta_{i,j}$ then 
\begin{gather}
\omega:=\frac12\sum_i ( b_i^+(-\tfrac32)b_i^-(-\tfrac12) - b_i^+(-\tfrac12)b_i^-(-\tfrac32) )\vv
\end{gather}
is a Virasoro element for ${\uA(\gt{b})}$, with central charge $c=-\frac12\dim\gt{b}$, that makes ${\uA(\gt{b})}$ a super vertex operator algebra. Note that although $\uA(\gt{b})$ is simple and $C_2$-cofinite, it is not rational (cf. \cite{MR2274534}). 

Set $\jmath:=\sum_i b_i^+(-\tfrac12)b_i^-(-\tfrac12)\vv$. Write $J(n)$ for the coefficient of $z^{-n-1}$ in $Y(\jmath,z)$ or $Y_\tw(\jmath,z)$, and write $L(n)$ for the coefficient of $z^{-n-2}$ in $Y(\omega,z)$ or $Y_\tw(\omega,z)$. Then, just as in the Clifford case, 
$J(0)$ and $L(0)$ commute, and act semisimply on $\uA(\gt{b})$ and ${\uA(\gt{b})}_\tw$, with finite-dimensional (simultaneous) eigenspaces. For the corresponding bigraded dimensions we have 
\begin{gather}
	\label{eqn:va:weyl-bgdim}
	\tr(y^{J(0)}q^{L(0)-\frac{c}{24}}|\uA(\gt{b}))
	=q^{\frac{d}{48}}\prod_{n>0}(1-y^{-1}q^{n-\frac12})^{-\frac d2 }(1-yq^{n-\frac12})^{-\frac d2},\\
	\label{eqn:va:weyl-bgdimtw}
	\tr(y^{J(0)}q^{L(0)-\frac{c}{24}}|\uA(\gt{b})_\tw)
	=y^{-\frac d 4}q^{-\frac{d}{24}}\prod_{n>0}(1-y^{-1}q^{n-1})^{-\frac d2 }(1-yq^{n})^{-\frac d2},
\end{gather}
when $d=\dim\gt{b}$. Note that $(1-X)^{-1}$ should be interpreted as $\sum_{n\geq 0}X^n$ in (\ref{eqn:va:weyl-bgdim}) and (\ref{eqn:va:weyl-bgdimtw}).

Similar again to the Clifford case, the group $\GL(\gt{b}^+)$ acts naturally on $\uA(\gt{b})$ and ${\uA(\gt{b})}_\tw$, preserving their bigradings. Explicitly, for $g\in \GL(\gt{b}^+)$ write $g$ also for the linear automorphism $(g,g')$ on $\gt{b}=\gt{b}^+\oplus \gt{b}^-$, where $g'\in\GL(\gt{b}^-)$ is determined by requiring $\llab gb,g'b'\rab=\llab b,b'\rab$ for all $b\in\gt{b}^+$ and $b'\in \gt{b}^-$. 
The action of $\GL(\gt{b}^+)$ on $\uA(\gt{b})$ is given by $g\cdot b_1(r_1)\dots b_n(r_n)\vv:=(gb_1)(r_1)\dots (gb_n)(r_n)\vv$ for $b_i\in\gt{b}$ and $r_i\in \ZZ+\frac12$, and similarly for $\uA(\gt{b})_\tw$. Vertex operators are preserved by this action, as are $\omega$ and $\jmath$, just as in \S\ref{sec:va:cliff}.

%----------------------------------------------------------------------------------------
\section{Meromorphic Jacobi Forms}\label{sec:jf}
%----------------------------------------------------------------------------------------

We briefly review the  
relationship between meromorphic Jacobi forms and the mock modular forms of umbral moonshine in this section. 
The original reference for this is \S4 of \cite{MUM}.
We refer the reader to \cite{Dabholkar:2012nd} for more detailed and more general discussions of mock modular forms, mock Jacobi forms and 
meromorphic Jacobi forms. 

Let $X$ be a Niemeier root system. For simplicity we restrict to the {\em pure type A} case that $X=A_{m-1}^{\oplus d}$ for some integer $m>1$ such that $m-1$ is a divisor of $24$, and $d:=\frac{24}{m-1}$. Let $N^{(m)}$ be the corresponding Niemeier lattice, and set $G^{(m)}:=\Aut(N^{(m)})/\Inn(N^{(m)})$ where $\Inn(N^{(m)})$ is the subgroup of $\Aut(N^{(m)})$ generated by reflections in root vectors. Then umbral moonshine \cite{UM,MUM,mumcor} attaches a $2m$-vector-valued mock modular form $H^{(m)}_g(\tau)=(H^{(m)}_{g,r}(\tau))_{r\xmod 2m}$ to each $g\in G^{(m)}$. 

One way to explain what this means is as follows. Let $\HH:=\{\tau\in\CC\mid \Im(\tau)>0\}$ denote the upper half-plane, and set $S:=\{(\tau,a\tau+b)\in\HH\times \CC\mid a,b\in\ZZ\}$. Define functions $\mu_{m,0}^k$ on $\HH\times \CC\setminus S$ for $k\xmod 2$ by setting $\mu_{m,0}^k(\tau,z):=\frac12(\mu_{m,0}(\tau,z)+(-1)^k\mu_{m,0}(\tau,z+\frac12))$, where 
\begin{gather}\label{eqn:jf-mum0}
	\mu_{m,0}(\tau,z):=\sum_{\ell\in\ZZ}y^{2m\ell}q^{m\ell^2}\frac{yq^\ell+1}{yq^\ell-1}
\end{gather}
for $y=e^{2\pi i z}$ and $q=e^{2\pi i\tau}$. Also define $\theta_{m,r}(\tau,z):=\sum_{\ell=r\xmod 2m}y^\ell q^{\frac{\ell^2}{4m}}$ for $r\xmod 2m$. Then for $\bar\chi_g^{(m)}$ and ${\chi}^{(m)}_g$ the characters of $G^{(m)}$ defined in \S B.2 of \cite{UM} or \cite{MUM} (or \S C of this work, for $m\in\{4,5,7,13\}$), the function 
\begin{gather}\label{eqn:jf:psig}
	\psi^{(m)}_g(\tau,z):=
	-\chi^{(m)}_g\mu_{m,0}^0(\tau,z)
	-\bar\chi^{(m)}_g\mu_{m,0}^1(\tau,z)
	+\sum_{r\xmod 2m}H^{(m)}_{g,r}(\tau)\theta_{m,r}(\tau,z)
\end{gather}
is a meromorphic Jacobi form with simple poles in $\ZZ\tau+\ZZ\frac12$. That is to say, we have $\psi^{(m)}_g=\frac{\phi_1}{\phi_2}$ for some (holomorphic) Jacobi forms $\phi_1$ and $\phi_2$, and for any fixed $\tau\in \HH$, the function $z\mapsto \psi^{(m)}_g(\tau,z)$ is meromorphic on $\CC$. Moreover, its poles are simple, and lie within the lattice $\ZZ\tau+\ZZ\frac12$. 

In the next section we will recover series expansions of the functions $\psi^{(m)}_g$ as traces on twisted modules for explicitly constructed super vertex algebras, for all $g\in G^{(m)}$ for $m=7$ (see \S\ref{sec:moon:7}) and $m=13$ (see \S\ref{sec:moon:13}), and for all $g$ in a maximal subgroup of $G^{(m)}$ for $m=4$ (see \S\ref{sec:moon:4}) and $m=5$ (see \S\ref{sec:moon:5}). This will solve the meromorphic module problem for umbral moonshine for the root systems $A_6^{\oplus 4}$ and $A_{12}^{\oplus 2}$, and partially solve it for $A_3^{\oplus 8}$ and $A_4^{\oplus 6}$.

%----------------------------------------------------------------------------------------
\section{Moonshine Modules}\label{sec:moon}
%----------------------------------------------------------------------------------------

We now present our main constructions. 

\subsection{Lambency Seven}\label{sec:moon:7}

Let $\gt{e}$ and $\gt{a}$ be $2$-dimensional complex vector spaces equipped with non-degenerate symmetric bilinear forms, and let 
$\gt{b}$ be a $4$-dimensional complex vector space equipped with a non-degenerate anti-symmetric bilinear form. 
Fix polarizations $\gt{e}=\gt{e}^+\oplus\gt{e}^-$, $\gt{a}=\gt{a}^+\oplus \gt{a}^-$ and $\gt{b}=\gt{b}^+\oplus\gt{b}^-$, and let $\{e^\pm\}$, $\{a^\pm\}$ and $\{b_i^\pm\}$ be bases for $\gt{e}^\pm$, $\gt{a}^\pm$ and $\gt{b}^\pm$, respectively, such that $\lab e^-,e^+\rab=\lab a^-,a^+\rab=1$ and $\llab b_i^-,b_j^+\rab=\delta_{i,j}$. 
Applying the constructions of \S\ref{sec:va} we obtain a super vertex operator algebra $W^{(7)}$, and a canonically twisted module for it $W^{(7)}_\tw$ by setting
\begin{gather}\label{eqn:moon:7-WWtw}
\begin{split}
	W^{(7)}&:=A(\gt{e})\otimes A(\gt{a})\,\otimes\uA(\gt{b}), \\
	W^{(7)}_\tw&:=A(\gt{e})_\tw\otimes A(\gt{a})_\tw\,\otimes \uA(\gt{b})_\tw,
\end{split}
\end{gather}
and equipping $W^{(7)}$ with the usual tensor product Virasoro element $\omega^{(7)}:=\omega\otimes\vv\otimes\vv+\vv\otimes\omega\otimes\vv+\vv\otimes\vv\otimes\omega$. To define bigradings on both spaces we set 
\begin{gather}
\jmath^{(7)}:=4\vv\otimes\jmath\otimes\vv+\vv\otimes\vv\otimes \jmath
\end{gather}
where $\jmath$ is defined for $A(\gt{a})$ and $\uA(\gt{b})$ as in \S\ref{sec:va}. .
We also define $\jmath_\gt{e}:=\jmath\otimes\vv\otimes\vv$.
Then $\GL(\gt{e}^+)\otimes \GL(\gt{a}^+)\otimes\GL(\gt{b}^+)$ acts naturally on $W^{(7)}$ and $W^{(7)}_\tw$, respecting the super vertex operator algebra module structures and preserving the bigradings. 

\begin{table}[ht]
\begin{center}
\caption{Eigenvalues for $G^{(7)}$}
\smallskip
\begin{small}
\begin{tabular}{ c |   cc } \toprule\label{tab:moon:7-evals}
    $[g]$ & $\lambda$ &$\{\ulambda_j\}$ \\ \midrule
    1A & $1$& $\{1,1\}$\\ 
    2A &  $1$& $\{-1,-1\}$\\ 
    4A & $1$&$\{i,-i\}$\\
    3A & $\omega$& $\{1,\omega\}$\\ 
    6A &$\omega^2$& $\{-1,-\omega^2\}$\\
    3B & $\omega^2$& $\{1,\omega^2\}$\\ 
    6B &$\omega$& $\{-1, -\omega\}$\\
   \bottomrule   
\end{tabular}
\end{small}
\end{center}
\end{table}

The character table of the umbral group $G^{(7)}$ is Table \ref{tab:irr:7} in \S\ref{sec:tab}. Choose homomorphisms $\varrho:G^{(7)}\to \GL(\gt{a}^+)$ and $\uvarrho\;\;:G^{(7)}\to \GL(\gt{b}^+)$ such that the corresponding characters are $\chi_2$ and $\chi_6$ in Table \ref{tab:irr:7}, respectively. 
Since $\chi_6$ is faithful the assignment $g\mapsto I\otimes \varrho(g)\,\otimes \uvarrho(g)$ defines faithful actions of $G^{(7)}$ on $W^{(7)}$ and $W^{(7)}_\tw$. Set $(-1)^F:=(-I)\otimes(-I)\otimes I$, and let $J_\gt{e}(0)$ denote the coefficient of $z^{-1}$ in $Y_\tw(\jmath_\gt{e},z)$. Let $J(0)$ be the coefficient of $z^{-1}$ in $Y_\tw(\jmath^{(7)},z)$, and let $L(0)$ be the coefficient of $z^{-2}$ in $Y_\tw(\omega^{(7)},z)$. For $g\in G^{(7)}$ define a formal series $\widetilde{\psi}^{(7)}_g\in \CC[y][[y^{-1}]][[q]]$ by setting
\begin{gather}\label{eqn:moon:7-trg}
	\widetilde{\psi}^{(7)}_g:=-\tr((g+g^{-1})J_\gt{e}(0)(-1)^Fy^{J(0)}q^{L(0)}|W^{(7)}_\tw).
\end{gather}

\begin{theorem}\label{thm:moon:7-psig}
For $g\in G^{(7)}$ the series $\widetilde{\psi}^{(7)}_g$ is the expansion of $\psi^{(7)}_g$ in the domain $0<-\Im(z)<\Im(\tau)$.
\end{theorem}

\begin{proof}
Let $g\in G^{(7)}$. The action of $g$ on $\gt{a}^+$ is multiplication by a scalar, $\lambda$ say, and there are a pair of eigenvalues $\{\ulambda_1,\ulambda_2\}$ for its action on $\gt{b}^+$. With this notation we have
\begin{gather}	
	\begin{split}
	\label{eqn:moon:7-psitildeprod}
\widetilde\psi^{(7)}_g= &-y 
\prod_{n>0} 
	\frac {(1-q^n)^2(1-\bar\lambda y^{-4}q^{n-1})(1-  \lambda y^{4}q^n)} 
	{\prod_{j=1}^2(1-\bar\ulambda_j  y^{-1} q^{n-1})(1- \ulambda_j \!y q^n)}\\
&-y
\prod_{n>0} 
	\frac {(1-q^n)^2(1- \lambda y^{-4}q^{n-1})(1- \bar\lambda y^{4}q^n)} 
	{\prod_{j=1}^2(1-\ulambda_j\!y^{-1} q^{n-1})(1- \bar\ulambda_j y q^n)}
		\end{split}
\end{gather}
where $(1-X)^{-1}$ is shorthand for $\sum_{k\geq 0}X^k$. This series converges in the given domain once we substitute $q=e^{2\pi i \tau}$ and $y=e^{2\pi i z}$, so we require to check that the right-hand side of (\ref{eqn:moon:7-psitildeprod}) agrees with the meromorphic Jacobi form $\psi^{(7)}_g$ when viewed as a function of $\tau$ and $z$. This follows from a case by case check using the values of $\lambda$ and $\ulambda_j$ in Table \ref{tab:moon:7-evals} and the explicit descriptions of the $\psi^{(7)}_g$ in (\ref{eqn:exp:4A6-psiXg}). For example, for $g\in 4A$ the right-hand side of (\ref{eqn:moon:7-psitildeprod}) becomes
\begin{gather}
	\label{eqn:moon:7-psitildeprod4A}
-2y 
\prod_{n>0} 
	\frac {(1-q^n)^2(1-y^{-4}q^{n-1})(1- y^{4}q^n)} 
	{(1+ y^{-2} q^{2n-2})(1+y^2 q^{2n})}
=-2i\frac{\eta(2\tau)\eta(\tau)\theta_1(\tau,4z)}{\theta_2(2\tau,2z)}
\end{gather}
which 
is exactly the expression for $\psi^{(7)}_{4A}$ that appears in (\ref{eqn:exp:4A6-psiXg}). The other cases are similar.
\end{proof}

\subsection{Lambency Thirteen}\label{sec:moon:13}

Let $\gt{e}$ and $\gt{a}$ be $2$-dimensional complex vector spaces equipped with non-degenerate symmetric bilinear forms, and let 
$\gt{b}$ and $\gt{b}'$ be $2$-dimensional complex vector spaces equipped with non-degenerate anti-symmetric bilinear forms. 
Fix polarizations $\gt{e}=\gt{e}^+\oplus\gt{e}^-$, $\gt{a}=\gt{a}^+\oplus\gt{a}^-$, $\gt{b}=\gt{b}^+\oplus\gt{b}^-$ and $\gt{b}'={\gt{b}'}^{+}\oplus {\gt{b}'}^-$, and let $\{e^\pm\}$, $\{a^\pm\}$, $\{b^\pm\}$ and $\{{b'}^{\pm}\}$ be bases for $\gt{e}^\pm$, $\gt{a}^\pm$, $\gt{b}^\pm$ and ${\gt{b}'}^\pm$, respectively, such that $\lab e^-,e^+\rab=\lab a^-,a^+\rab=\llab b^-,b^+\rab=\llab {b'}^-,{b'}^+\rab=1$.

Define a super vertex operator algebra $W^{(13)}$, and a canonically twisted $W^{(13)}$-module $W^{(13)}_\tw$ by setting
\begin{gather}\label{eqn:moon:13-WWtw}
\begin{split}
	W^{(13)}&:=A(\gt{e})\otimes A(\gt{a})\,\otimes\uA(\gt{b})\otimes \uA(\gt{b}'), \\
	W^{(13)}_\tw&:=A(\gt{e})_\tw\otimes A(\gt{a})_\tw\,\otimes \uA(\gt{b})_\tw\otimes \uA(\gt{b}')_\tw.
\end{split}
\end{gather}
Equip $W^{(13)}$ with the usual tensor product Virasoro element, denote it $\omega^{(13)}$, 
set $\jmath_\gt{e}:=\jmath\otimes\vv\otimes\vv\otimes\vv$, and set 
\begin{gather}
\jmath^{(13)}:=6\vv\otimes\jmath\otimes\vv\otimes\vv+\vv\otimes\vv\otimes \jmath\otimes\vv+3\vv\otimes\vv\otimes\vv\otimes\jmath.
\end{gather}
Then $\GL(\gt{e}^+)\otimes \GL(\gt{a}^+)\otimes\GL(\gt{b}^+)\otimes \GL({\gt{b}'}^+)$ acts naturally on $W^{(13)}$ and $W^{(13)}_\tw$ respecting the super vertex operator algebra module structures and preserving the bigradings. 

The 
umbral group $G^{(13)}$ is cyclic of order $4$. (Cf. Table \ref{tab:irr:13}.) Define compatible actions of $G^{(13)}$ on $W^{(13)}$ and $W^{(13)}_\tw$ by choosing a generator and mapping it to $I\otimes (-I)\otimes (iI)\otimes (-iI)$ in $\GL(\gt{e}^+)\otimes \GL(\gt{a}^+)\otimes\GL(\gt{b}^+)\otimes \GL({\gt{b}'}^+)$. Similar to \S\ref{sec:moon:7} we set
$(-1)^F:=(-I)\otimes(-I)\otimes I$, 
let $J_\gt{e}(0)$ denote the coefficient of $z^{-1}$ in $Y_\tw(\jmath_\gt{e},z)$, let $J(0)$ be the coefficient of $z^{-1}$ in $Y_\tw(\jmath^{(13)},z)$, and let $L(0)$ be the coefficient of $z^{-2}$ in $Y_\tw(\omega^{(13)},z)$. Then for $g\in G^{(13)}$ we define a formal series 
in $\CC[y][[y^{-1}]][[q]]$ by setting
\begin{gather}\label{eqn:moon:13-trg}
	\widetilde{\psi}^{(13)}_g:=-\tr((g+g^{-1})J_\gt{e}(0)(-1)^Fy^{J(0)}q^{L(0)}|W^{(13)}_\tw).
\end{gather}

\begin{theorem}\label{thm:moon:13-psig}
For $g\in G^{(13)}$ the series $\widetilde{\psi}^{(13)}_g$ is the expansion of $\psi^{(13)}_g$ in the domain $0<-\Im(z)<\Im(\tau)$.
\end{theorem}

\begin{proof}
Let $g\in G^{(13)}$. Then $g$ acts by scalar multiplication on $\gt{a}^+$, $\gt{b}^+$ and ${\gt{b}'}^+$. Let $\lambda$, $\ulambda$ and $\ulambdap$ be the respective scalars. Then we have
\begin{gather}	
	\begin{split}
	\label{eqn:moon:13-psitildeprod}
\widetilde\psi^{(13)}_g= 
&-y 
\prod_{n>0} 
	\frac 
	{ (1-q^n)^2(1-\bar\lambda y^{-6} q^{n-1})(1-\lambda y^6 q^n)} 
	{ (1-\bar \ulambda \,y^{-1}q^{n-1})(1- \ulambda yq^n) (1-\bar \ulambdap y^{-3}q^{n-1})(1- \ulambdap \!y^{3}q^n) }\\
&-y
\prod_{n>0} 
	\frac 
	{ (1-q^n)^2(1-\lambda y^{-6} q^{n-1})(1-\bar\lambda y^6 q^n)} 
	{ (1-\ulambda y^{-1}q^{n-1})(1- \bar\ulambda\, yq^n) (1- \ulambdap \!y^{-3}q^{n-1})(1- \bar\ulambdap y^{3}q^n) }
		\end{split}
\end{gather}
where, as before, $(1-X)^{-1}$ is shorthand for $\sum_{k\geq 0}X^k$. This convergence of this series, upon substituting 
$q=e^{2\pi i \tau}$ and $y=e^{2\pi i z}$, is the same as in Theorem \ref{thm:moon:7-psig}. So we just need to check that the right-hand side of (\ref{eqn:moon:7-psitildeprod}) agrees with the meromorphic Jacobi form $\psi^{(13)}_g$ when viewed as a function of $\tau$ and $z$. This follows from a case by case comparison with (\ref{eqn:exp:2A12-psiXg}). For example, for $g$ the involution in $G^{(13)}$ the right-hand side of (\ref{eqn:moon:13-psitildeprod}) becomes
\begin{gather}
	\label{eqn:moon:13-psitildeprod2A}
	\begin{split}
&-2y 
\prod_{n>0} 
	\frac {(1-q^n)^2(1-y^{-6}q^{n-1})(1- y^{6}q^n)} 
	{(1+ y^{-1} q^{n-1})(1+y q^{n})(1+ y^{-3} q^{n-1})(1+y^3 q^{n})}\\
=&-2i\frac{\eta(\tau)^3\theta_1(\tau,6z)}{\theta_2(\tau,z)\theta_2(\tau,3z)}
	\end{split}
\end{gather}
which 
is precisely $\psi^{(13)}_{2A}$ as it appears in (\ref{eqn:exp:2A12-psiXg}). We leave the remaining cases to the reader.
\end{proof}

\subsection{Lambency Four}\label{sec:moon:4}

This section and the next are similar to the previous two, except that we realize umbral moonshine only for maximal subgroups $G^{(4)}_{336}$ and $G^{(5)}_{24}$ of the umbral groups $G^{(4)}$ and $G^{(5)}$. 

For $\ell=4$ let $\gt{e}$ be just as in \S\S\ref{sec:moon:7},\ref{sec:moon:13}, let $\gt{a}$ be a $6$-dimensional complex vector space equipped with a non-degenerate symmetric bilinear form, and let $\gt{b}$ be an $8$-dimensional complex vector space equipped with a non-degenerate anti-symmetric bilinear form. Choose polarizations $\gt{e}=\gt{e}^+\oplus\gt{e}^-$, $\gt{a}=\gt{a}^+\oplus\gt{a}^-$ and $\gt{b}=\gt{b}^+\oplus\gt{b}^-$, and let $\{e^\pm\}$, $\{a^\pm_i\}$ and $\{b_i^\pm\}$ be bases for $\gt{e}^\pm$, $\gt{a}^\pm$ and $\gt{b}^\pm$, respectively, such that $\lab e^-,e^+\rab=1$ and $\lab a_i^-,a_j^+\rab=\llab b_i^-,b_j^+\rab=\delta_{i,j}$.
Define a super vertex operator algebra and a canonically twisted module for it by setting
\begin{gather}\label{eqn:moon:4-WWtw}
\begin{split}
	W^{(4)}&:=A(\gt{e})\otimes A(\gt{a})\,\otimes\uA(\gt{b}), \\
	W^{(4)}_\tw&:=A(\gt{e})_\tw\otimes A(\gt{a})_\tw\,\otimes \uA(\gt{b})_\tw,
\end{split}
\end{gather}
and let $\omega^{(4)}$ denote the (tensor product) Virasoro element for $W^{(4)}$. 
Set $\jmath_\gt{e}:=\jmath\otimes\vv\otimes\vv$ and
$\jmath^{(4)}:=2\vv\otimes\jmath\otimes\vv+\vv\otimes\vv\otimes \jmath$.
As in \S\ref{sec:moon:7}, the group $\GL(\gt{e}^+)\otimes \GL(\gt{a}^+)\otimes\GL(\gt{b}^+)$ acts naturally on $W^{(4)}$ and $W^{(4)}_\tw$, respecting the super vertex operator algebra module structures and preserving the bigradings defined by the zero modes of $\omega^{(4)}$ and $\jmath^{(4)}$.

\begin{table}[ht]
\begin{center}
\caption{Eigenvalues for $\ell=4$}
\smallskip
\begin{small}
\begin{tabular}{ c | c  c c} \toprule\label{tab:moon:4-evals}
    $[g]$ & $\{\lambda_i\}$ & $\{\ulambda_j\}$ \\ \midrule
    1A & $\{1,1,1\}$ & $\{1,1,1,1\}$\\ 
    2A & $\{1,1,1\}$ & $\{-1,-1,-1,-1\}$\\ 
    4A & $\{1,-1,-1\}$&$\{\ii,\ii,-\ii,-\ii\}$\\
    3A & $\{1,\omega,\omega^2\}$ & $\{1,1,\omega,\omega^2\}$\\ 
    6A & $\{1,\omega,\omega^2\}$ & $\{-1,-1,-\omega,-\omega^2\}$\\
    8A & $\{1,\ii,-\ii\}$ & $\{\zeta_8,\zeta_8^3,\zeta_8^5,\zeta_8^7\}$\\ 
    7A & $\{\zeta_7,\zeta_7^2,\zeta_7^4\}$ & $\{1,\zeta_7,\zeta_7^2,\zeta_7^4\}$\\
    7B & $\{\zeta_7^3,\zeta_7^5,\zeta_7^6\}$ & $\{1,\zeta_7^3,\zeta_7^5,\zeta_7^6\}$\\
    14A & $\{\zeta_7,\zeta_7^2,\zeta_7^4\}$ & $\{-1,-\zeta_7,-\zeta_7^2,-\zeta_7^4\}$\\ 
    14B & $\{\zeta_7^3,\zeta_7^5,\zeta_7^6\}$ & $\{-1,-\zeta_7^3,-\zeta_7^5,-\zeta_7^6\}$\\ 
    \bottomrule   
\end{tabular}
\end{small}
\end{center}
\end{table}

We write $G^{(4)}_{336}$ for a subgroup of $G^{(4)}$ isomorphic to $\SL_2(7)$. Such subgroups are maximal and unique up to conjugacy, but note that there are two other conjugacy classes of maximal subgroups of order $336$. The character tables of $G^{(4)}$ and $G^{(4)}_{336}$ are Tables \ref{tab:chars:irr:4} and \ref{tab:chars:irr:4_336}, respectively. Table \ref{tab:chars:irr:4_336} also gives the fusion of conjugacy classes with respect to an embedding $\iota: G^{(4)}_{336}\to G^{(4)}$.

Choose homomorphisms $\varrho:G^{(4)}_{336}\to \GL(\gt{a}^+)$ and $\uvarrho\;\;:G^{(4)}_{336}\to \GL(\gt{b}^+)$ such that the corresponding characters are $\chi_2$ and $\chi_8$ in Table \ref{tab:chars:irr:4_336}, respectively. 
Then the assignment $g\mapsto I\otimes \varrho(g)\,\otimes \uvarrho(g)$ defines faithful and compatible actions of $G^{(4)}_{336}$ on $W^{(4)}$ and $W^{(4)}_\tw$. Define $(-1)^F$, $J_\gt{e}(0)$, $J(0)$ and $L(0)$ just as in \S\S\ref{sec:moon:7},\ref{sec:moon:13}. For $g\in G^{(4)}_{336}$ we consider the formal series $\widetilde{\psi}^{(4)}_g$ defined in direct analogy with (\ref{eqn:moon:7-trg}) and (\ref{eqn:moon:13-trg}),
\begin{gather}\label{eqn:moon:4-trg}
	\widetilde{\psi}^{(4)}_g:=-\tr((g+g^{-1})J_\gt{e}(0)(-1)^Fy^{J(0)}q^{L(0)}|W^{(4)}_\tw).
\end{gather}

\begin{proposition}\label{prop:moon:4-psig}
For $g\in G^{(4)}_{336}$ the series $\widetilde{\psi}^{(4)}_g$ is the expansion of $\psi^{(4)}_g$ in the domain $0<-\Im(z)<\Im(\tau)$.
\end{proposition}

\begin{proof}
The proof is very similar to that of Theorem \ref{thm:moon:7-psig}.
Let $g\in G^{(4)}_{336}$. Let $\{\lambda_i\}$ and $\{\ulambda_j\}$ be the eigenvalues for the actions of $g$ on $\gt{a}^+$ and $\gt{b}^+$, respectively. Then we have
\begin{gather}	
	\begin{split}
	\label{eqn:moon:4-psitildeprod}
\widetilde\psi^{(4)}_g= &-y 
\prod_{n>0} 
	\frac {(1-q^n)^2\prod_{i=1}^3(1-\bar \lambda_i y^{-2}q^{n-1})(1- \lambda_i y^{2}q^n)} 
	{\prod_{j=1}^4(1-\bar\ulambda_j y^{-1} q^{n-1})(1-\ulambda_j \!y q^n)}\\
&-y
\prod_{n>0} 
	\frac {(1-q^n)^2\prod_{i=1}^3(1-\lambda_i y^{-2}q^{n-1})(1- \bar\lambda_i y^{2}q^n)} 
	{\prod_{j=1}^4(1-\ulambda_j \!y^{-1} q^{n-1})(1-\bar\ulambda_j y q^n)}.
		\end{split}
\end{gather}
As in the proof of Theorem \ref{thm:moon:7-psig} we just require to check that the right-hand side of (\ref{eqn:moon:4-psitildeprod}) agrees with the meromorphic Jacobi form $\psi^{(4)}_g$ when viewed as a function of $\tau$ and $z$, and we achieve this by comparing in each case the values of $\lambda_i$ and $\ulambda_j$ in Table \ref{tab:moon:4-evals} with the explicit descriptions of the $\psi^{(4)}_g$ in (\ref{eqn:exp:8A3-psiXg}). 
\end{proof}

\subsection{Lambency Five}\label{sec:moon:5}

Let $\gt{e}$, $\gt{a}$ and $\gt{a}'$ be $2$-dimensional complex vector spaces equipped with non-degenerate symmetric bilinear forms, and let 
$\gt{b}$ be a $6$-dimensional complex vector space equipped with a non-degenerate anti-symmetric bilinear form. 
Fix polarizations $\gt{e}=\gt{e}^+\oplus\gt{e}^-$, $\gt{a}=\gt{a}^+\oplus\gt{a}^-$, $\gt{a}'={\gt{a}'}^+\oplus{\gt{a}'}^-$ and $\gt{b}=\gt{b}^+\oplus \gt{b}^-$, and let $\{e^\pm\}$, $\{a^\pm\}$, $\{{a'}^\pm\}$ and $\{b^{\pm}_i\}$ be bases for $\gt{e}^\pm$, $\gt{a}^\pm$, ${\gt{a}'}^\pm$ and $\gt{b}^\pm$, respectively, such that $\lab e^-,e^+\rab=\lab a^-,a^+\rab=\lab {a'}^-,{a'}^+\rab=1$ and $\llab b^-_i,b^+_j\rab=\delta_{i,j}$.
Similar to (\ref{eqn:moon:7-WWtw}), (\ref{eqn:moon:13-WWtw}) and (\ref{eqn:moon:4-WWtw}) we define a super vertex operator algebra 
and a canonically twisted 
module for it by setting
\begin{gather}\label{eqn:moon:5-WWtw}
\begin{split}
	W^{(5)}&:=A(\gt{e})\otimes A(\gt{a})\,\otimes A(\gt{a}')\,\otimes \uA(\gt{b}), \\
	W^{(5)}_\tw&:=A(\gt{e})_\tw\otimes A(\gt{a})_\tw\,\otimes A(\gt{a}')_\tw\,\otimes \uA(\gt{b}')_\tw.
\end{split}
\end{gather}
Equip $W^{(5)}$ with the usual tensor product Virasoro element, 
set $\jmath_\gt{e}:=\jmath\otimes\vv\otimes\vv\otimes\vv$ and 
$\jmath^{(5)}:=2\vv\otimes\jmath\otimes\vv\otimes\vv+3\vv\otimes\vv\otimes \jmath\otimes\vv+\vv\otimes\vv\otimes\vv\otimes\jmath$.
Then $\GL(\gt{e}^+)\otimes \GL(\gt{a}^+)\otimes\GL({\gt{a}'}^+)\otimes \GL(\gt{b}^+)$ acts naturally on $W^{(5)}$ and $W^{(5)}_\tw$, respecting the super vertex operator algebra module structures and preserving the bigradings defined by the Virasoro element and $\jmath^{(5)}$. 

\begin{table}[ht]
\begin{center}
\caption{Eigenvalues for $\ell=5$}
\smallskip
\begin{small}
\begin{tabular}{ c | c  cc } \toprule\label{tab:moon:5-evals}
    $[g]$ & $\lambda$ & $\lambda'$ &$\{\mu_j\}$ \\ \midrule
    1A &$1$ &$1$& $\{1,1,1\}$\\ 
    2A & $1$ &$-1$& $\{-1,-1,-1\}$\\ 
    2B & $1$&$1$& $\{1,-1,-1\}$\\
    2C & $1$&$-1$& $\{1,1,-1\}$\\
    3A & $1$&$1$& $\{1,\omega,\omega^2\}$\\ 
    6A & $1$&$-1$& $\{-1,-\omega,-\omega^2\}$\\
    4A &$-1$ &$\ii$&$\{\ii,-\ii,-\ii\}$\\
    4B &$-1$ &$-\ii$&$\{\ii,\ii,-\ii\}$\\
    12A &$-1$ &$\ii$&$\{-\ii\omega,\ii,-\ii\omega^2\}$ \\
    12B & $-1$&$-\ii$&$\{\ii\omega,-\ii,\ii\omega^2\}$ \\\bottomrule
\end{tabular}
\end{small}
\end{center}
\end{table}

There 
is a unique conjugacy class of maximal subgroups of $G^{(5)}$ with order $24$. We choose a subgroup in this class, denote it $G^{(5)}_{24}$, and let $g\mapsto \iota g$ denote the inclusion $G^{(5)}_{24}\to G^{(5)}$.
The character tables of $G^{(5)}$ and $G^{(5)}_{24}$ are Tables \ref{tab:chars:irr:5} and \ref{tab:irr:5_24}, respectively, and Table \ref{tab:irr:5_24} gives the fusion of conjugacy classes under $g\mapsto \iota g$.
Choose homomorphisms $\varrho:G^{(5)}_{24}\to \GL(\gt{a}^+)$, $\varrho':G^{(5)}\to \GL({\gt{a}'}^+)$ and $\uvarrho\;\,:G^{(5)}\to \GL(\gt{b}^+)$ such that the corresponding characters are $\chi_2$, $\chi_3$ and $\chi_7+\chi_{12}$ in Table \ref{tab:irr:5_24}, respectively. Define 
$(-1)^F$, $J_\gt{e}(0)$, $J(0)$ and $L(0)$ as in \S\S\ref{sec:moon:7}-\ref{sec:moon:4}, and to $g\in G^{(5)}_{24}$ attach the formal series 
\begin{gather}\label{eqn:moon:5-trg}
	\widetilde{\psi}^{(5)}_g:=-\tr((g+g^{-1})J_\gt{e}(0)(-1)^Fy^{J(0)}q^{L(0)}|W^{(5)}_\tw).
\end{gather}

\begin{proposition}\label{prop:moon:5-psig}
For $g\in G^{(5)}$ the series $\widetilde{\psi}^{(5)}_g$ is the expansion of $\psi^{(5)}_g$ in the domain $0<-\Im(z)<\Im(\tau)$.
\end{proposition}
\begin{proof}
The proof is directly similar to the proofs of Theorems \ref{thm:moon:7-psig} and \ref{thm:moon:13-psig}, and Proposition \ref{prop:moon:4-psig}, and depends upon a verification that the natural product representations of the $\widetilde{\psi}^{(5)}_g$ for $g\in G^{(5)}_{24}$ coincide with the meromorphic Jacobi forms $\psi^{(5)}_{g}$ given in (\ref{eqn:exp:6A4-psiXg}). For the convenience of the reader we present the eigenvalues arising from the representations $\varrho$, $\varrho'$ and $\uvarrho\;$ in Table \ref{tab:moon:5-evals}.
\end{proof}

%---------------------------------------------------------------------------------------%
\section*{Acknowledgements}
%---------------------------------------------------------------------------------------%

The authors thank Matthias Gaberdiel 
and Miranda Cheng for comments, and discussions on closely related topics. 
J.D. gratefully acknowledges support from the Simons Foundation (\#316779), and the U.S. National Science Foundation (DMS 1203162, DMS 1601306).

%----------------------------------------------------------------------------------------
\appendix
%----------------------------------------------------------------------------------------

%----------------------------------------------------------------------------------------
\section{Character Tables}\label{sec:tab}
%----------------------------------------------------------------------------------------

Here we give character tables 
for the groups that 
appear in \S\ref{sec:moon}.
We use the abbreviations $a_n:=\sqrt{-n}$, $b_n:=(-1+\sqrt{-n})/2$ and $r_n:=\sqrt{n}$.

\begin{table}[ht]
\begin{center}
\caption{Character table of ${G}^{(7)}\simeq\SL_2(3)$}\label{tab:irr:7}
\smallskip
\begin{tabular}{c|rrrrrrr}\toprule
$[g]$	&   1A&   2A&   4A&   3A&   6A&   3B&   6B\\
	\midrule
$\chi_1$&   $1$&   $1$&   $1$&   $1$&   $1$&   $1$&   $1$\\
$\chi_2$&   $1$&   $1$&   $1$&   ${b_3}$&   $\overline{b_3}$&   $\overline{b_3}$&   ${b_3}$\\
$\chi_3$&   $1$&   $1$&   $1$&   $\overline{b_3}$&   ${b_3}$&   ${b_3}$&   $\overline{b_3}$\\
$\chi_4$&   $3$&   $3$&   $-1$&   $0$&   $0$&   $0$&   $0$\\
$\chi_5$&   $2$&   $-2$&   $0$&   $-1$&   $1$&   $-1$&   $1$\\
$\chi_6$&   $2$&   $-2$&   $0$&   $-\overline{b_3}$&   ${b_3}$&   $-{b_3}$&   $\overline{b_3}$\\
$\chi_7$&   $2$&   $-2$&   $0$&   $-{b_3}$&   $\overline{b_3}$&   $-\overline{b_3}$&   ${b_3}$\\\bottomrule
\end{tabular}
\end{center}
\end{table}

\begin{table}[ht]
\begin{center}
\caption{Character table of ${G}^{(13)}\simeq 4$}\label{tab:irr:13}
\smallskip
\begin{tabular}{c|rrrr}\toprule
$[g]$&   1A&   2A&   4A&   4B\\
	\midrule
$\chi_1$&   $1$&   $1$&   $1$&   $1$\\
$\chi_2$&   $1$&   $1$&   $-1$&   $-1$\\
$\chi_3$&   $1$&   $-1$&   $a_1$&   $\overline{a_1}$\\
$\chi_4$&   $1$&   $-1$&   $\overline{a_1}$&   $a_1$\\\bottomrule
\end{tabular}
\end{center}
\end{table}

\begin{table}[ht]
\begin{center}
\caption{Character table of ${G}^{(4)}\simeq 2.\AGL_3(2)$}\label{tab:chars:irr:4}

\smallskip
\begin{tabular}{c|r@{\;\;}r@{\;\;}r@{\;\;}r@{\;\;}r@{\;\;}r@{\;\;}r@{\;\;}r@{\;\;}r@{\;\;}r@{\;\;}r@{\;\;}r@{\;\;}r@{\;\;}r@{\;\;}r@{\;\;}r} \toprule
$[g]$&   1A&   2A&   2B&   2C&   4A&   4B&   3A&   6A&   6B&   6C&   4C&   8A&   7A&   7B&   14A&   14B\\
	\midrule	
${\chi}_{1}$&   $1$&   $1$&   $1$&   $1$&   $1$&   $1$&   $1$&   $1$&   $1$&   $1$&   $1$&   $1$&   $1$&   $1$&   $1$&   $1$\\
${\chi}_{2}$&   $3$&   $3$&   $3$&   $-1$&   $-1$&   $-1$&   $0$&   $0$&   $0$&   $0$&   $1$&   $1$&   ${b_7}$&   $\overline{b_7}$&   ${b_7}$&   $\overline{b_7}$\\
${\chi}_{3}$&   $3$&   $3$&   $3$&   $-1$&   $-1$&   $-1$&   $0$&   $0$&   $0$&   $0$&   $1$&   $1$&   $\overline{b_7}$&   ${b_7}$&   $\overline{b_7}$&   ${b_7}$\\
${\chi}_{4}$&   $6$&   $6$&   $6$&   $2$&   $2$&   $2$&   $0$&   $0$&   $0$&   $0$&   $0$&   $0$&   $-1$&   $-1$&   $-1$&   $-1$\\
${\chi}_{5}$&   $7$&   $7$&   $7$&   $-1$&   $-1$&   $-1$&   $1$&   $1$&   $1$&   $1$&   $-1$&   $-1$&   $0$&   $0$&   $0$&   $0$\\
${\chi}_{6}$&   $8$&   $8$&   $8$&   $0$&   $0$&   $0$&   $-1$&   $-1$&   $-1$&   $-1$&   $0$&   $0$&   $1$&   $1$&   $1$&   $1$\\
${\chi}_{7}$&   $7$&   $7$&   $-1$&   $3$&   $-1$&   $-1$&   $1$&   $1$&   $-1$&   $-1$&   $1$&   $-1$&   $0$&   $0$&   $0$&   $0$\\
${\chi}_{8}$&   $7$&   $7$&   $-1$&   $-1$&   $3$&   $-1$&   $1$&   $1$&   $-1$&   $-1$&   $-1$&   $1$&   $0$&   $0$&   $0$&   $0$\\
${\chi}_{9}$&   $14$&   $14$&   $-2$&   $2$&   $2$&   $-2$&   $-1$&   $-1$&   $1$&   $1$&   $0$&   $0$&   $0$&   $0$&   $0$&   $0$\\
${\chi}_{10}$&   $21$&   $21$&   $-3$&   $-3$&   $1$&   $1$&   $0$&   $0$&   $0$&   $0$&   $1$&   $-1$&   $0$&   $0$&   $0$&   $0$\\
${\chi}_{11}$&   $21$&   $21$&   $-3$&   $1$&   $-3$&   $1$&   $0$&   $0$&   $0$&   $0$&   $-1$&   $1$&   $0$&   $0$&   $0$&   $0$\\
${\chi}_{12}$&   $8$&   $-8$&   $0$&   $0$&   $0$&   $0$&   $2$&   $-2$&   $0$&   $0$&   $0$&   $0$&   $1$&   $1$&   $-1$&   $-1$\\
${\chi}_{13}$&   $8$&   $-8$&   $0$&   $0$&   $0$&   $0$&   $-1$&   $1$&   ${a_3}$&   $\overline{a_3}$&   $0$&   $0$&   $1$&   $1$&   $-1$&   $-1$\\
${\chi}_{14}$&   $8$&   $-8$&   $0$&   $0$&   $0$&   $0$&   $-1$&   $1$&   $\overline{a_3}$&   ${a_3}$&   $0$&   $0$&   $1$&   $1$&   $-1$&   $-1$\\
${\chi}_{15}$&   $24$&   $-24$&   $0$&   $0$&   $0$&   $0$&   $0$&   $0$&   $0$&   $0$&   $0$&   $0$&   ${b_7}$&   $\overline{b_7}$&   $-b_7$&   $-\overline{b_7}$\\
${\chi}_{16}$&   $24$&   $-24$&   $0$&   $0$&   $0$&   $0$&   $0$&   $0$&   $0$&   $0$&   $0$&   $0$&   $\overline{b_7}$&   ${b_7}$&   $-\overline{b_7}$&   $-b_7$\\
\bottomrule
\end{tabular}
\end{center}
\end{table}

\begin{table}[ht]
\begin{center}
\caption{Character table of ${G}^{(4)}_{336}\simeq \SL_2(7)$}\label{tab:chars:irr:4_336}

\smallskip
\begin{tabular}{c|rrrrrrrrrrr} \toprule
$[g]$&   1A&   2A&   4A&   3A&   6A&   8A&   8B&   7A&   7B&   14A&   14B\\
	\midrule	
$[\iota{g}]$&1A&	2A&	4A&	3A&	6A&	8A&	8A&	7A&	7B&	14A&	14B	\\
	\midrule
${\chi}_{1}$&   $1$&   $1$&   $1$&   $1$&   $1$&   $1$&   $1$&   $1$&   $1$&   $1$&   $1$\\
${\chi}_{2}$&   $3$&   $3$&   $-1$&   $0$&   $0$&   $1$&   $1$&   ${b_7}$&   $\overline{b_7}$&   ${b_7}$&   $\overline{b_7}$\\
${\chi}_{3}$&   $3$&   $3$&   $-1$&   $0$&   $0$&   $1$&   $1$&   $\overline{b_7}$&   ${b_7}$&   $\overline{b_7}$&   ${b_7}$\\
${\chi}_{4}$&   $6$&   $6$&   $2$&   $0$&   $0$&   $0$&   $0$&   $-1$&   $-1$&   $-1$&   $-1$\\
${\chi}_{5}$&   $7$&   $7$&   $-1$&   $1$&   $1$&   $-1$&   $-1$&   $0$&   $0$&   $0$&   $0$\\
${\chi}_{6}$&   $8$&   $8$&   $0$&   $-1$&   $-1$&   $0$&   $0$&   $1$&   $1$&   $1$&   $1$\\
${\chi}_{7}$&   $4$&   $-4$&   $0$&   $1$&   $-1$&   $0$&   $0$&   $-b_7$&   $-\overline{b_7}$&   ${b_7}$&   $\overline{b_7}$\\
${\chi}_{8}$&   $4$&   $-4$&   $0$&   $1$&   $-1$&   $0$&   $0$&   $-\overline{b_7}$&   $-b_7$&   $\overline{b_7}$&   ${b_7}$\\
${\chi}_{9}$&   $6$&   $-6$&   $0$&   $0$&   $0$&   $r_{2}$&   $-r_{2}$&   $-1$&   $-1$&   $1$&   $1$\\
${\chi}_{10}$&   $6$&   $-6$&   $0$&   $0$&   $0$&   $-r_2$&   $r_{2}$&   $-1$&   $-1$&   $1$&   $1$\\
${\chi}_{11}$&   $8$&   $-8$&   $0$&   $-1$&   $1$&   $0$&   $0$&   $1$&   $1$&   $-1$&   $-1$\\
\bottomrule
\end{tabular}
\end{center}
\end{table}

\begin{table}[ht]
\begin{center}
\caption{Character table of ${G}^{(5)}\simeq \GL_2(5)/2$}\label{tab:chars:irr:5}
\smallskip
\begin{tabular}{c|r@{\;\;\;}r@{\;\;\;}r@{\;\;\;}r@{\;\;\;}r@{\;\;\;}r@{\;\;\;}r@{\;\;\;}r@{\;\;\;}r@{\;\;\;}r@{\;\;\;}r@{\;\;\;}r@{\;\;\;}r@{\;\;\;}r}\toprule
$[g]$&   1A&   2A&   4A&   4B&   2B&   2C&   4C&   4D&   3A&   6A&   12A&   12B&   5A&   10A\\
	\midrule
${\chi}_{1}$&   $1$&   $1$&   $1$&   $1$&   $1$&   $1$&   $1$&   $1$&   $1$&   $1$&   $1$&   $1$&   $1$&   $1$\\
${\chi}_{2}$&   $1$&   $1$&   $-1$&   $-1$&   $1$&   $1$&   $-1$&   $-1$&   $1$&   $1$&   $-1$&   $-1$&   $1$&   $1$\\
${\chi}_{3}$&   $1$&   $-1$&   ${a_1}$&   $\overline{a_1}$&   $1$&   $-1$&   ${a_1}$&   $\overline{a_1}$&   $1$&   $-1$&   ${a_1}$&   $\overline{a_1}$&   $1$&   $-1$\\
${\chi}_{4}$&   $1$&   $-1$&   $\overline{a_1}$&   ${a_1}$&   $1$&   $-1$&   $\overline{a_1}$&   ${a_1}$&   $1$&   $-1$&   $\overline{a_1}$&   ${a_1}$&   $1$&   $-1$\\
${\chi}_{5}$&   $4$&   $4$&   $2$&   $2$&   $0$&   $0$&   $0$&   $0$&   $1$&   $1$&   $-1$&   $-1$&   $-1$&   $-1$\\
${\chi}_{6}$&   $4$&   $4$&   $-2$&   $-2$&   $0$&   $0$&   $0$&   $0$&   $1$&   $1$&   $1$&   $1$&   $-1$&   $-1$\\
${\chi}_{7}$&   $4$&   $-4$&   $2{a_1}$&   $2\overline{a_1}$&   $0$&   $0$&   $0$&   $0$&   $1$&   $-1$&   $\overline{a_1}$&   ${a_1}$&   $-1$&   $1$\\
${\chi}_{8}$&   $4$&   $-4$&   $2\overline{a_1}$&   $2{a_1}$&   $0$&   $0$&   $0$&   $0$&   $1$&   $-1$&   ${a_1}$&   $\overline{a_1}$&   $-1$&   $1$\\
${\chi}_{9}$&   $5$&   $5$&   $1$&   $1$&   $1$&   $1$&   $-1$&   $-1$&   $-1$&   $-1$&   $1$&   $1$&   $0$&   $0$\\
${\chi}_{10}$&   $5$&   $5$&   $-1$&   $-1$&   $1$&   $1$&   $1$&   $1$&   $-1$&   $-1$&   $-1$&   $-1$&   $0$&   $0$\\
${\chi}_{11}$&   $5$&   $-5$&   ${a_1}$&   $\overline{a_1}$&   $1$&   $-1$&   $\overline{a_1}$&   ${a_1}$&   $-1$&   $1$&   ${a_1}$&   $\overline{a_1}$&   $0$&   $0$\\
${\chi}_{12}$&   $5$&   $-5$&   $\overline{a_1}$&   ${a_1}$&   $1$&   $-1$&   ${a_1}$&   $\overline{a_1}$&   $-1$&   $1$&   $\overline{a_1}$&   ${a_1}$&   $0$&   $0$\\
${\chi}_{13}$&   $6$&   $6$&   $0$&   $0$&   $-2$&   $-2$&   $0$&   $0$&   $0$&   $0$&   $0$&   $0$&   $1$&   $1$\\
${\chi}_{14}$&   $6$&   $-6$&   $0$&   $0$&   $-2$&   $2$&   $0$&   $0$&   $0$&   $0$&   $0$&   $0$&   $1$&   $-1$\\
\bottomrule
\end{tabular}
\end{center}
\end{table}

\begin{table}[ht]
\begin{center}
\caption{Character table of ${G}^{(5)}_{24}\simeq S_3\times 4$}\label{tab:irr:5_24}
\smallskip
\begin{tabular}{c|rrrrrrrrrrrr}\toprule
$[g]$&   1A&   2A&   4A&   4B&   2B&   2C&   4C&   4D&   3A&   6A&   12A&   12B\\
	\midrule
$[\iota{g}]$&1A&	2A&	4A&	4B&	2B&	2C&	4A&	4B&	3A&	6A&	12A&	12B	\\
	\midrule
${\chi}_{1}$&   $1$&   $1$&   $1$&   $1$&   $1$&   $1$&   $1$&   $1$&   $1$&   $1$&   $1$&   $1$\\
${\chi}_{2}$&   $1$&   $1$&   $-1$&   $-1$&   $1$&   $1$&   $-1$&   $-1$&   $1$&   $1$&   $-1$&   $-1$\\
${\chi}_{3}$&   $1$&   $-1$&   $a_1$&   $\overline{a_1}$&   $1$&   $-1$&   $a_1$&   $\overline{a_1}$&   $1$&   $-1$&   $a_1$&   $\overline{a_1}$\\
${\chi}_{4}$&   $1$&   $-1$&   $\overline{a_1}$&   $a_1$&   $1$&   $-1$&   $\overline{a_1}$&   $a_1$&   $1$&   $-1$&   $\overline{a_1}$&   $a_1$\\
${\chi}_{5}$&   $1$&   $1$&   $1$&   $1$&   $-1$&   $-1$&   $-1$&   $-1$&   $1$&   $1$&   $1$&   $1$\\
${\chi}_{6}$&   $1$&   $1$&   $-1$&   $-1$&   $-1$&   $-1$&   $1$&   $1$&   $1$&   $1$&   $-1$&   $-1$\\
${\chi}_{7}$&   $1$&   $-1$&   $a_1$&   $\overline{a_1}$&   $-1$&   $1$&   $\overline{a_1}$&   $a_1$&   $1$&   $-1$&   $a_1$&   $\overline{a_1}$\\
${\chi}_{8}$&   $1$&   $-1$&   $\overline{a_1}$&   $a_1$&   $-1$&   $1$&   $a_1$&   $\overline{a_1}$&   $1$&   $-1$&   $\overline{a_1}$&   $a_1$\\
${\chi}_{9}$&   $2$&   $2$&   $2$&   $2$&   $0$&   $0$&   $0$&   $0$&   $-1$&   $-1$&   $-1$&   $-1$\\
${\chi}_{10}$&   $2$&   $2$&   $-2$&   $-2$&   $0$&   $0$&   $0$&   $0$&   $-1$&   $-1$&   $1$&   $1$\\
${\chi}_{11}$&   $2$&   $-2$&   $2a_1$&   $2\overline{a_1}$&   $0$&   $0$&   $0$&   $0$&   $-1$&   $1$&   $\overline{a_1}$&   $a_1$\\
${\chi}_{12}$&   $2$&   $-2$&   $2\overline{a_1}$&   $2a_1$&   $0$&   $0$&   $0$&   $0$&   $-1$&   $1$&   $a_1$&   $\overline{a_1}$\\
\bottomrule
\end{tabular}
\end{center}
\end{table}

\clearpage

\section{Umbral Jacobi Forms}\label{sec:ujf}

Here we recall from \S B of \cite{umrec} the meromorphic Jacobi forms associated to the groups that we construct in \S\ref{sec:moon}. Some of these expressions were obtained earlier in \cite{UM,MUM}. To present the formulas we use the Dedekind eta and Jacobi theta functions,
\begin{gather}
	\eta(\tau):=q^{\frac1{24}}\prod_{n>0}(1-q^n),\\
	\theta_1(\tau,z):=-iq^{\frac18}y^{\frac12}\prod_{n>0}(1-y^{-1}q^{n-1})(1-yq^n)(1-q^n),\\
	\theta_2(\tau,z):=q^{\frac18}y^{\frac12}\prod_{n>0}(1+y^{-1}q^{n-1})(1+yq^n)(1-q^n),
\end{gather}
where $q=e^{2\pi i \tau}$ and $y=e^{2\pi i z}$. In what follows, the subscript in $\psi^{(\ell)}_{nZ}$ names a conjugacy class in the umbral group $G^{(\ell)}$, where the labelling of the conjugacy classes is as defined by the character tables in \S\ref{sec:tab}. For the cases that $\ell=4$ and $\ell=5$ we only recall formulas for the conjugacy classes that are represented by elements of the groups $G^{(4)}_{336}$ and $G^{(5)}_{24}$ (cf. \S\ref{sec:moon:4} and \S\ref{sec:moon:5}).

\begin{gather}\label{eqn:exp:4A6-psiXg}	
\begin{split}
	\psi^{(7)}_{1A}(\tau,z)&:=2i\frac{\eta(\tau)^3\theta_1(\tau,4z)}{\theta_1(\tau,z)^{2}}\\
	\psi^{(7)}_{2A}(\tau,z)&:=-2i\frac{\eta(\tau)^3\theta_1(\tau,4z)}{\theta_2(\tau,z)^{2}}\\
	\psi^{(7)}_{4A}(\tau,z)&:=-2i\frac{\eta(\tau)\eta(2\tau)\theta_1(\tau,4z)}{\theta_2(2\tau,2z)}\\
	\psi^{(7)}_{3A}(\tau,z)&:=
	-i\frac{\eta(3\tau)}{\theta_1(3\tau,3z)}\times\\
	&\left(\theta_1(\tau,4z+\tfrac13)\theta_1(\tau,z-\tfrac13)
	+\theta_1(\tau,4z-\tfrac13)\theta_1(\tau,z+\tfrac13)\right)
	\\
	\psi^{(7)}_{6A}(\tau,z)&:=
	-i\frac{\eta(3\tau)}{\theta_2(3\tau,3z)}\times\\
	&\left(\theta_1(\tau,4z+\tfrac13)\theta_1(\tau,z-\tfrac16)
	+\theta_1(\tau,4z-\tfrac13)\theta_1(\tau,z+\tfrac16)\right)	
\end{split}
\end{gather}

\begin{gather}\label{eqn:exp:2A12-psiXg}	
\begin{split}	
	\psi^{(13)}_{1A}(\tau,z)&:=2i\frac{\eta(\tau)^3\theta_1(\tau,6z)}{\theta_1(\tau,z)\theta_1(\tau,3z)}\\
	\psi^{(13)}_{2A}(\tau,z)&:=-2i\frac{\eta(\tau)^3\theta_1(\tau,6z)}{\theta_2(\tau,z)\theta_2(\tau,3z)}\\
	\psi^{(13)}_{4AB}(\tau,z)&:=
	-i\frac{\eta(2\tau)^2\theta_2(\tau,6z)}{\eta(\tau)\theta_2(2\tau,2z)\theta_2(2\tau,6z)}\times\\
	&\left(\theta_1(\tau,z+\tfrac14)\theta_1(\tau,3z+\tfrac14)
	-\theta_1(\tau,z-\tfrac14)\theta_1(\tau,3z-\tfrac14)\right)
\end{split}
\end{gather}

\begin{gather}\label{eqn:exp:8A3-psiXg}	
	\begin{split}
	\psi^{(4)}_{1A}&:=2i\frac{\eta(\tau)^3\theta_1(\tau,2z)^3}{\theta_1(\tau,z)^{4}}\\
	\psi^{(4)}_{2A}&:=2i\frac{\eta(\tau)^3\theta_1(\tau,2z)^3}{\theta_2(\tau,z)^{4}}\\
	\psi^{(4)}_{4A}&:=-2i\frac{\eta(2\tau)^2\theta_1(\tau,2z)\theta_2(\tau,2z)^2}{\eta(\tau)\theta_2(2\tau,2z)^{2}}\\
	\psi^{(4)}_{3A}&:=2i\frac{\eta(\tau)^3\theta_1(3\tau,6z)}{\theta_1(\tau,z)\theta_1(3\tau,3z)}\\
	\psi^{(4)}_{6A}&:=-2i\frac{\eta(\tau)^3\theta_1(3\tau,6z)}{\theta_2(\tau,z)\theta_2(3\tau,3z)}\\
	\psi^{(4)}_{8A}&:=-2i\frac{\eta(\tau)\eta(4\tau)\theta_1(\tau,2z)\theta_2(2\tau,4z)}{\eta(2\tau)\theta_2(4\tau,4z)}\\
	\psi^{(4)}_{7AB}&:=-i\frac{\eta(7\tau)}{\eta(\tau)^4\theta_1(7\tau,7z)}\times\\
	(
	\theta_1(\tau,2z+\tfrac{1}{7})&\theta_1(\tau,2z+\tfrac{2}{7})\theta_1(\tau,2z+\tfrac{4}{7})
	\theta_1(\tau,z-\tfrac{1}{7})\theta_1(\tau,z-\tfrac{2}{7})\theta_1(\tau,z-\tfrac{4}{7})\\
	+\theta_1(\tau,2z-&\tfrac{1}{7})\theta_1(\tau,2z-\tfrac{2}{7})\theta_1(\tau,2z-\tfrac{4}{7})
	\theta_1(\tau,z+\tfrac{1}{7})\theta_1(\tau,z+\tfrac{2}{7})\theta_1(\tau,z+\tfrac{4}{7})
	)
	\\
	\psi^{(4)}_{14AB}&:=i\frac{\eta(7\tau)}{\eta(\tau)^4\theta_2(7\tau,7z)}\times\\
	(
	\theta_1(\tau,2z+\tfrac{1}{7})&\theta_1(\tau,2z+\tfrac{2}{7})\theta_1(\tau,2z+\tfrac{4}{7})\theta_2(\tau,z-\tfrac{1}{7})\theta_2(\tau,z-\tfrac{2}{7})\theta_2(\tau,z-\tfrac{4}{7})\\
	+\theta_1(\tau,2z-&\tfrac{1}{7})\theta_1(\tau,2z-\tfrac{2}{7})\theta_1(\tau,2z-\tfrac{4}{7})\theta_2(\tau,z+\tfrac{1}{7})\theta_2(\tau,z+\tfrac{2}{7})\theta_2(\tau,z+\tfrac{4}{7}))
	\end{split}
\end{gather}

\begin{gather}\label{eqn:exp:6A4-psiXg}	
\begin{split}
	\psi^{(5)}_{1A}(\tau,z)&:=2i\frac{\eta(\tau)^3\theta_1(\tau,2z)\theta_1(\tau,3z)}{\theta_1(\tau,z)^{3}}\\
	\psi^{(5)}_{2A}(\tau,z)&:=-2i\frac{\eta(\tau)^3\theta_1(\tau,2z)\theta_2(\tau,3z)}{\theta_2(\tau,z)^{3}}\\
	\psi^{(5)}_{2B}(\tau,z)&:=-2i\frac{\eta(\tau)^3\theta_1(\tau,2z)\theta_1(\tau,3z)}{\theta_1(\tau,z)\theta_2(\tau,z)^{2}}\\
	\psi^{(5)}_{2C}(\tau,z)&:=2i\frac{\eta(\tau)^3\theta_1(\tau,2z)\theta_2(\tau,3z)}{\theta_1(\tau,z)^{2}\theta_2(\tau,z)}\\
	\psi^{(5)}_{3A}(\tau,z)&:=-2i\frac{\eta(3\tau)\theta_1(\tau,2z)\theta_1(\tau,3z)}{\theta_1(3\tau,3z)}\\
	\psi^{(5)}_{6A}(\tau,z)&:=-2i\frac{\eta(3\tau)\theta_1(\tau,2z)\theta_2(\tau,3z)}{\theta_2(3\tau,3z)}\\
	\psi^{(5)}_{4AB}(\tau,z)&:=
	-i\frac{\eta(2\tau)^2\theta_2(\tau,2z)}{\eta(\tau)\theta_2(2\tau,2z)^2}\times\\
	&\left(\theta_1(\tau,z+\tfrac14)\theta_1(\tau,3z+\tfrac14)-\theta_1(\tau,z-\tfrac14)\theta_1(\tau,3z-\tfrac14)\right)
	\\
	\psi^{(5)}_{12AB}(\tau,z)&:=
	i\frac{\eta(6\tau)\theta_2(\tau,2z)}{\eta(\tau)^3\theta_2(6\tau,6z)}\times\\
	&(\theta_1(\tau,z+\tfrac{1}{12})\theta_1(\tau,z+\tfrac{1}{4})\theta_1(\tau,z+\tfrac{5}{12})\theta_1(\tau,3z-\tfrac14)
	\\
	&\quad-
	\theta_1(\tau,z-\tfrac{1}{12})\theta_1(\tau,z-\tfrac{1}{4})\theta_1(\tau,z-\tfrac{5}{12})\theta_1(\tau,3z+\tfrac14))
\end{split}
\end{gather}

\clearpage

\section{Euler Characters}\label{sec:tab:eul}

Here we tabulate the character values $\bar{\chi}^{(\ell)}_{g}$ and $\chi^{(\ell)}_g$ 
for each $g\in G^{(\ell)}$, for each $\ell\in\{4,5,7,13\}$.

\begin{table}[ht]

\begin{center}
\caption{
Euler characters at $\ell=7$}\label{tab:chars:7}\smallskip
\begin{tabular}{c|ccccc}\toprule
$[g]$&   1A&   2A&   4A&   3AB&   6AB\\ 
	\midrule
$\bar{\chi}^{(\ell)}_{g}$&	4&	4&	0&	1&	1\\
$\chi^{(\ell)}_{g}$&	4&	-4&	0&	1&	-1\\
\bottomrule
\end{tabular}

\medskip

\caption{Euler characters at $\ell=13$}\label{tab:chars:13}
\smallskip
\begin{tabular}{r|rrr}\toprule
	$[g]$&	1A&	2A&	4AB\\
		\midrule
	$\bar{\chi}^{(\ell)}_{g}$	&2&2&0\\
	$\chi^{(\ell)}_{g}$	&2&-2&0\\
\bottomrule
\end{tabular}

\medskip

\caption{Euler characters at $\ell=4$}\label{tab:chars:4}
\smallskip
\begin{tabular}{c@{\, }|@{\;}c@{\, }c@{\, }c@{\, }c@{\, }c@{\, }c@{\, }c@{\, }c@{\, }c@{\, }c@{\, }c@{\, }c@{\, }c}\toprule
$[g]$&   		1A&   2A&   	2B&   	4A&			4B&			2C&   	3A&   	6A&   		6BC&   	8A&   	4C&   	7AB&   	14AB\\ 
	\midrule
$\bar{\chi}^{(\ell)}_g$&   $8$&$8$&	$0$& 	$0$& 		$0$&		$4$&  	$2$& 	$2$&  		$0$& 	$0$& 	$2$& 	$1$& 	$1$\\
$\chi^{(\ell)}_g$&   $8$&$-8$&	$0$&	$0$& 		$0$&		$0$&  	$2$& 	$-2$& 		$0$& 	$0$& 	$0$& 	$1$& 	$-1$\\
\bottomrule
\end{tabular}

\medskip

\caption{Euler characters at $\ell=5$}\label{tab:chars:5}
\smallskip
\begin{tabular}{c@{\, }|@{\;}c@{\, }c@{\, }c@{\, }c@{\, }c@{\, }c@{\, }c@{\, }c@{\, }c@{\, }c@{\, }c@{\, }c@{\, }c}\toprule
$[g]$&   		1A&		2A&   	2B&   	2C&			3A&			6A&   	5A&   	10A&   		4AB&   	4CD&	12AB\\ 
	\midrule
$\bar{\chi}^{(\ell)}_{g}$&   $6$&	$6$&	$2$& 	$2$& 		$0$&		$0$&  	$1$& 	$1$&  		$0$& 	$2$& 	$0$ 	\\
$\chi^{(\ell)}_{g}$&   $6$&	$-6$&	$-2$&	$2$& 		$0$&		$0$&  	$1$& 	$-1$& 		$0$&	$0$& 	$0$ 	\\
\bottomrule
\end{tabular}
\smallskip
\end{center}

\end{table}

\clearpage

%------------------------------------------------------------------%

%------------------------------------------------------------------%

\end{document}